\documentclass[a4paper]{amsart}

\usepackage{graphicx}
\usepackage{amssymb}
\usepackage{amsmath}
\usepackage{amsthm}
\usepackage{amscd}
\usepackage[all,2cell]{xy}

\UseAllTwocells
\SilentMatrices

\theoremstyle{plain}
\newtheorem{thm}{Theorem}[section]
\newtheorem{prop}[thm]{Propsition}
\newtheorem{cor}[thm]{Corollary}
\newtheorem{lem}[thm]{Lemma}
\theoremstyle{definition}

\theoremstyle{remark}

\newcommand{\To}{\xrightarrow}

\newcommand{\rad}{\mathrm{rad}}
\newcommand{\soc}{\mathrm{soc}}

\newcommand{\Natural}{\mathbb{N}}

\begin{document}
\title[A note on resolution quivers]{A note on resolution quivers}
\author[Dawei Shen]{Dawei Shen}
\thanks{Supported by the National Natural Science Foundation of China (No. 11201446).}
\date{November 16, 2012.}
\thanks{E-mail: sdw12345$\symbol{64}$mail.ustc.edu.cn}
\dedicatory{}
\commby{}
\begin{abstract}
  Recently, Ringel introduced the resolution quiver for a connected Nakayama algebra. It is known that each connected component of the resolution quiver has a unique cycle. We prove that all cycles in the resolution quiver are of the same size. We introduce the notion of weight for a cycle in the resolution quiver. It turns out that all cycles have the same weight.
\end{abstract}
\maketitle

\section{Introduction}
Let $A$ be a connected Nakayama algebra without simple projective modules. All modules are left modules of finite length. We denote the number of simple $A$-modules by $n(A)$. Let $\gamma(S)=\tau\soc{P(S)}$ for a simple $A$-module $S$ \cite{R1}, where $P(S)$ is the projective cover of $S$ and $\tau=\mathrm{DTr}$ is the Auslander-Reiten translation \cite{ARS}. Ringel \cite{R1} defined the {\em resolution quiver} $R(A)$ of $A$ as follows: the vertices correspond to simple $A$-modules and there is an arrow from $S$ to $\gamma(S)$ for each simple $A$-module $S$. The resolution quiver gives a fast algorithm to decide whether $A$ is a Gorenstein algebra or not, and whether it is CM-free or not; see \cite{R1}.

Using the map $f$ introduced in \cite{G}, the notion of resolution quiver applies to any connected Nakayama algebra. It is known that each connected component of $R(A)$ has a unique cycle.

Let A be a connected Nakayama algebra and $C$ be a cycle in $R(A)$. Assume that the vertices of $C$ are $S_1, S_2, \cdots , S_m$. We define the {\em weight} of $C$ to be $\frac{\sum_{k=1}^mc_k}{n(A)}$, where $c_k$ is the length of the projective cover of $S_k$. The aim of this note is to prove the following result.

\begin{prop} \label{prop 1.1}
Let $A$ be a connected Nakayama algebra. Then all cycles in its resolution quiver are of the same size and of the same weight.
\end{prop}

As a consequence of Proposition \ref{prop 1.1}, if the resolution quiver has a loop, then all cycles are loops; this result is obtained by Ringel \cite{R1, R2}. The proof of Proposition \ref{prop 1.1} uses {\em left retractions} of Nakayama algebras studied in \cite{CY}.

\section{The proof of Proposition 1.1}
Let $A$ be a connected Nakayama algebra. Recall that $n=n(A)$ is the number of simple $A$-modules. Let $S_1,S_2,\cdots,S_n$ be a complete set of pairwise non-isomorphic simple $A$-modules and $P_i$ be the projective cover of $S_i$. We require that $\rad{P_i}$ is a factor module of $P_{i+1}$. Here,  we identify $n+1$ with $1$.

Recall that $\mathbf{c}(A)=(c_1,c_2,\cdots,c_n)$ is an {\em admissible sequence} for $A$, where $c_i$ is the length of $P_i$; see \cite[Chapter IV. 2]{ARS}. We denote $p(A)=\min\{c_1,c_2,\cdots,c_n\}$. The algebra $A$ is called a {\em line algebra} if $c_n=1$ or, equivalently, the valued quiver of $A$ is a line; otherwise, $A$ is called a {\em cycle algebra} or, equivalently, the valued quiver of $A$ is a cycle. Then $A$ is a cycle algebra if and only if $A$ has no simple projective modules.

Following \cite{G}, we introduce a map $f_A:\{1,2,\cdots,n\}\to\{1,2,\cdots,n\}$ such that $n$ divides $f_A(i)-(c_i+i)$ for $1\leq{i}\leq n$. The {\em resolution quiver} $R(A)$ of $A$ is defined as follows:  its vertices are $1,2,\cdots,n$ and there is an arrow from $i$ to $f_A(i)$. Observe that for a cycle algebra $A$ we have $\gamma(S_i)=S_{f_A(i)}$. Then by identifying $i$ with $S_i$, the resolution quiver $R(A)$ coincides with that in \cite{R1}.

Assume that $A$ is a cycle algebra which is not self-injective. After possible cyclic permutations, we may assume that its admissible sequence $\mathbf{c}(A)=(c_1,c_2,\cdots,c_n)$ is {\em normalized} \cite{CY}, that is, $p(A)=c_1=c_n-1$. Recall from \cite{CY} that there is an algebra homomorphism $\eta:A \to L(A)$ with $L(A)$ a connected Nakayama algebra such that its admissible sequence $\mathbf{c}(L(A)) = (c_1',c_2',\cdots,c_{n-1}')$ is given by $c_i'=c_i-[\frac{c_i+i-1}{n}]$ for $1\leq i\leq n-1$; in particular, $n(L(A))=n(A)-1$. Here, for a real number $x$, $[x]$ denotes the largest integer not greater than $x$.  The algebra homomorphism $\eta$ is called the {\em left retraction} \cite{CY} of $A$ with respect to $S_n$.

We introduce a map $\pi:\{1,2,\cdots,n\}\to\{1,2,\cdots,n-1\}$ such that $\pi(i)=i$ for $i<n$ and $\pi(n)=1$. The following result is contained in the proof of \cite[Lemma 3.7]{CY}.

\begin{lem} \label{lem 2.1}
Let $A$ be a cycle algebra which is not self-injective. Then $\pi f_A(i)=f_{L(A)}\pi(i)$ for $1\leq i\leq n$.
\end{lem}

\begin{proof}
Let $c_i+i=kn+j$ with $k\in\Natural$ and $1\leq j\leq n$. In particular, $f_A(i)=j$.

For $i<n$, we have
\begin{equation} \label{eq 1}
c_{\pi(i)}^\prime+i=c_i+i-\bigg[\frac{c_i+i-1}{n}\bigg]=kn+j-\bigg[\frac{kn+j-1}{n}\bigg]=k(n-1)+j.
\end{equation}
Then $\pi f_A(i)=\pi(j)$ and $f_{L(A)}\pi(i)=f_{L(A)}(i)=\pi(j)$.

For $i=n$, we have
\begin{equation} \label{eq 2}
c_{\pi(n)}^\prime+n=c_n-1+n-\bigg[\frac{c_n-1}{n}\bigg]=kn+j-1-\bigg[\frac{kn+j-n-1}{n}\bigg]=k(n-1)+j.
\end{equation}
Then $\pi f_A(n)=\pi(j)$ and $f_{L(A)}\pi(n)=f_{L(A)}(1)=\pi(j)$.
\end{proof}

The previous lemma gives rise to a unique morphism of resolution quivers
\[
\tilde\pi:R(A)\longrightarrow R(L(A))
\]
such that $\tilde\pi(i)=\pi(i)$. Then $\tilde\pi$ sends the unique arrow from $i$ to $f_A(i)$ to the unique arrow in $R(L(A))$ from $\pi(i)$ to $f_{L(A)}\pi(i)=\pi f_A(i)$. The morphism $\tilde\pi$ identifies the vertices $1$ and $n$ as well as the arrows starting from $1$ and $n$. Because $1$ and $n$ are in the same connected component of $R(A)$, we infer that $R(A)$ and $R(L(A))$ have the same number of connected components.

Let $A$ be a connected Nakayama algebra and $C$ be a cycle in $R(A)$. The {\em size} of $C$ is the number of vertices in $C$. We recall that the {\em weight} of $C$ is given by $w(C)=\frac{\sum_{k}c_k}{n(A)}$, where $k$ runs over all vertices in $C$. We mention that $w(C)$ is an integer; see \eqref{eq 3}. A vertex $x$ in $R(A)$ is said to be {\em cyclic} provided that $x$ belongs to a cycle.

\begin{lem} \label{lem 2.2}
Let $A$ be a cycle algebra which is not self-injective. Then $\tilde\pi$ induces a bijection between the set of cycles in $R(A)$ and the set of cycles in $R(L(A))$, which preserves sizes and weights.
\end{lem}

\begin{proof}
We observe that for two vertices $x$ and $y$ in $R(A)$, $\pi(x)=\pi(y)$ if and only if $x=y$ or $\{x,y\}=\{1,n\}$. Note that $f_A(1)=f_A(n)$. So the vertices $1$ and $n$ are in the same connected component of $R(A)$ and they are not cyclic at the same time.

Let $C$ be a cycle in $R(A)$ with vertices $x_1,x_2,\cdots,x_s$ such that $x_{i+1}=f_A(x_i)$. Here, we identify $s+1$ with $1$. Since the vertices $1$ and $n$ are not cyclic at the same time, we have that $\pi(x_1),\pi(x_2),\cdots,\pi(x_s)$ are pairwise distinct and $\tilde{\pi}(C)$ is a cycle in $R(L(A))$. Hence $\tilde\pi$ induces a map from the set of cycles in $R(A)$ to the set of cycles in $R(L(A))$. Obviously the map is injective. On the other hand, recall that $R(L(A))$ and $R(A)$ have the same number of connected components, thus they have the same number of cycles. Hence $\tilde\pi$ induces a bijection between the set of cycles in $R(A)$ and the set of cycles in $R(L(A))$ which preserves sizes.

It remains to prove that $w(C)=w(\tilde\pi(C))$. We assume that $c_{x_i}+x_i=k_i n+x_{i+1}$ with $k_i\in\Natural$. Then we have
\begin{equation} \label{eq 3}
w(C)=\frac{\sum_{i=1}^sc_{x_i}}{n}=\sum_{i=1}^s{k_i}.
\end{equation}
Recall that $n(L(A))=n-1$. We note that $c_{\pi(x_i)}^\prime+x_i= k_i(n-1)+x_{i+1}$; see \eqref{eq 1} and \eqref{eq 2}. Hence $\sum_{i=1}^sc_{\pi(x_i)}'=(n-1)\sum_{i=1}^sk_i$  and the assertion follows.
\end{proof}

Recall from \cite[Theorem 3.8]{CY} that there exists a sequence of algebra homomorphisms
\begin{equation} \label{eq 4}
A=A_0\To{\eta_0}A_1\To{\eta_1}A_2\to\cdots\to A_{r-1}\To{\eta_{r-1}}A_r
\end{equation}
such that each $A_i$ is a connected Nakayama algebra, $\eta_i:A_i\to A_{i+1}$ is a left retraction and $A_r$ is self-injective.

We now prove Proposition \ref{prop 1.1}.

\begin{proof}[\bf Proof of Proposition \ref{prop 1.1}:]
Assume that $A$ is a connected self-injective Nakayama algebra with $n(A)=n$ and  admissible sequence $\mathbf{c}(A)=(c,c,\cdots,c)$. Then a direct calculation shows that $R(A)$ consists entirely of cycles and each cycle is of size $\frac{n}{(n,c)}$ and of weight $\frac{c}{(n,c)}$, where $(n,c)$ is the greatest common divisor of $n$ and $c$. In particular, all cycles in $R(A)$ are of the same size and of the same weight.

In general, let $A$ be a connected Nakayama algebra whose admissible sequence is $\mathbf{c}(A)=(c_1,c_2,\cdots,c_n)$. Take $A'$ to be a connected Nakayama algebra with admissible sequence $\mathbf{c}(A')=(c_1+n,c_2+n,\cdots,c_n+n)$. Then $R(A)=R(A')$ and for any cycle $C$ in $R(A)$, the corresponding cycle $C'$ in $R(A')$ satisfies $w(C')=w(C)+s(C)$, where $s(C)$ denotes the size of $C$. The statement for $A$ holds if and only if it holds for $A'$.

We now assume that $A$ is a connected Nakayama algebra with $p(A)>n(A)$. One proves by induction that each $A_i$ in the sequence \eqref{eq 4} satisfies $p(A_i)>n(A_i)$. In particular, each $A_i$ is a cycle algebra. We can apply Lemma \ref{lem 2.2} repeatedly. Then the statement for $A$ follows from the statement for the self-injective Nakayama algebra $A_r$, which is already proved above.
\end{proof}

We conclude this note with a consequence of the above proof.
\begin{cor}
Let $A$ be a connected Nakayama algebra of infinite global dimension. Then we have the following statements.

$\mathrm{(1)}$ The number of cyclic vertices of the resolution quiver $R(A)$ equals the number of simple $A$-modules of infinite projective dimension.

$\mathrm{(2)}$ The number of simple $A$-modules of infinite projective dimension equals the number of simple $A$-modules of infinite injective dimension.
\end{cor}

\begin{proof}
$\mathrm{(1)}$ All the algebras $A_i$ in the sequence \eqref{eq 4} have infinite global dimension; see \cite[Lemma 2.4]{CY}. In particular, they are cycle algebras. We apply Lemma \ref{lem 2.2} repeatedly and obtain a bijection between the set of cyclic vertices of $R(A)$ and the set of cyclic vertices of $R(A_r)$. Recall that all vertices of $R(A_r)$ are cyclic, and $n(A_r)$ equals $n(A)$ minus the number of simple $A$-modules of finite projective dimension; see \cite[Theorem 3.8]{CY}. Then the statement follows immediately.

$\mathrm{(2)}$ Recall from \cite[Corollary 3.6]{M} that a simple $A$-module $S$ is cyclic in $R(A)$ if and only if $S$ has infinite injective dimension. Then $\mathrm{(2)}$ follows from $\mathrm{(1)}$.
\end{proof}

\section*{Acknowledgements}
The author thanks his supervisor Professor Xiao-Wu Chen for his guidance and Professor Claus Michael Ringel for his encouragement.

\vskip 10pt

{\footnotesize \noindent Dawei Shen \\
School of Mathematical Sciences, University of Science and Technology of China, Hefei, Anhui 230026, P. R. China \\
URL: http://home.ustc.edu.cn/$^\sim$sdw12345}
\end{document}